\documentclass[11pt]{article}
\usepackage{amsmath,amsthm,amssymb,amsfonts,amscd, amsxtra, mathrsfs,textcomp,verbatim,inputenc}
\usepackage{mathbbol}
\usepackage{url}
\usepackage{xcolor}
\usepackage{mathtools}
\usepackage{cite}
\usepackage{enumerate}
\usepackage[margin=2.5 cm,nohead]{geometry}
\theoremstyle{plain}
\theoremstyle{definition}
\newtheorem{theorem}{Theorem}
\newtheorem{lemma}{Lemma}

\newtheorem{corollary}{Corollary}
\newtheorem{proposition}{Proposition}

\DeclareMathOperator{\diag}{diag}

\DeclareMathOperator{\argmin}{argmin}
\DeclareMathOperator{\inte}{int}

\newcommand{\Sp}{\mathbb S^n}

\newcommand{\R}{\mathbb R}

\newcommand{\lng}{\left\langle}
\newcommand{\rng}{\right\rangle}
\newcommand{\lf}{\left}
\newcommand{\rg}{\right}

\newcommand{\ds}{\displaystyle}

\newcommand{\f}{\frac}
\newcommand{\tp}{^\top}
\newcommand{\prp}{^\perp}
\newcommand{\inv}{^{-1}}

\newcommand{\lm}{\lambda_{\min}}
\newcommand{\vo}{1}

\def\DHLhks#1#2{\setbox0=\hbox{$#1\sqrt#2$}\dimen0=\ht0
\advance\dimen0-0.2\ht0
\setbox2=\hbox{\vrule height\ht0 depth -\dimen0}%
{\box0\lower0.4pt\box2}}

\date{}

\begin{document}
\sffamily
\title{Why Study Spherical Convexity of Non-Homogeneous Quadratic Functions? What Makes It Surprising?}
\author{R. Bolton \thanks{School of Mathematics, University of Birmingham, Watson Building, Edgbaston, Birmingham - B15 2TT, United Kingdom
	(E-mail: {\tt ryanbolton@hotmail.co.uk})}
\and
S. Z. N\'emeth \thanks{School of Mathematics, University of Birmingham, Watson Building, Edgbaston, Birmingham - B15 2TT, United Kingdom
(E-mail: {\tt s.nemeth@bham.ac.uk})}
}
\maketitle
\begin{abstract}
	This paper establishes necessary, sufficient, and equivalent conditions for the 
	spherical convexity of non-homogeneous quadratic functions. By examining criteria for 
	determining spherical convexity, we identified unique properties that differentiate 
	spherically convex quadratic functions from their geodesically convex counterparts in 
	both hyperbolic and Euclidean spaces. Since spherically convex functions over the 
	entire sphere are constant, our analysis focuses on proper spherically convex subsets 
	of the sphere. Our primary results concern non-homogeneous quadratic functions on the 
	spherically convex set of unit vectors with positive coordinates. We also extend our 
	findings to more general spherically convex sets. Additionally, the paper explores 
	special cases of non-homogeneous quadratic functions where the defining matrix is of a 
	specific type, such as positive, diagonal, or a Z-matrix. This study not only provides 
	useful criteria for spherical convexity but also reveals surprising characteristics of 
	spherically convex quadratic functions, contributing to a deeper understanding of 
	convexity in spherical geometries.
\end{abstract}

\section{Introduction}

This study was prompted by a reviewer of the paper \cite{FerreiraNemethZhu2023}, who suggested analyzing the geodesic convexity of non-homogeneous quadratic functions on the sphere as a natural counterpart to the same problem on a hyperbolic manifold.

Initially, we anticipated implementing similar methodologies to those used in the aforementioned paper. However, we found that the approaches employed in our study needed to diverge significantly from those applied to the analogous problem on the hyperbolic manifold.

This disparity arises from our examination of the geodesic convexity of
non-homogeneous convex functions across the entire hyperbolic manifold, a problem
that becomes trivial on the sphere. Indeed, such a task presents no difficulty
on the sphere due to all its geodesics being closed, rendering
geodesically convex functions on the sphere constant. Consequently, we had
to confine our investigation to geodesically convex subsets of the sphere, which
we termed "spherically convex". As a result, the problem became significantly
more challenging since linear coordinate transformations, akin to those employed
in \cite{FerreiraNemethZhu2023}, failed to leave the considered spherically
convex set invariant.
The most intuitive spherically convex set is the collection of points on the sphere 
with nonnegative coordinates, i.e., the intersection of the sphere with the 
nonnegative orthant. While much of our study has focused on the interior of this set,
we have also derived results for a general spherically convex set, defined as the 
intersection of the sphere with a pointed convex cone. 

Motivations for studying the geodesic convexity of non-homogeneous spherically convex quadratic functions
are outlined below:

\begin{enumerate}
    \item Generally, the utility of exploring geodesic convexity mirrors that of Euclidean convexity.
\begin{enumerate}
    \item A local minimizer of a geodesic convex function is also a global minimizer, and strict geodesic convexity ensures a unique minimizer.
    \item In general optimization algorithms on Riemannian manifolds exhibit better convergence properties and superior performance when optimizing geodesically convex functions.
\end{enumerate}
\item   Numerous research papers and books (e.g.,
	\cite{Rubio2022,Ferreira2008,FerreiraLouzeiroPrudente2019,FerreiraLouzeiroPrudente2019b,Nguyen2019,Giesl2023,Rapcsak1997,Udriste1994,Wiesel2012,Boumal2020,Melvin2022,Maass2022,Yi2021,Burgisser2019,Sra2015,Allen-Zhu2018}) have leveraged the concept of geodesic convexity of functions on manifolds to establish various optimization results and analyze algorithmic convergence. 

It is imperative to exercise caution when employing theories based on geodesic convexity without a comprehensive understanding of its nuances.

For instance, several publications in reputable journals have made incorrect assumptions regarding geodesic convexity on Hadamard manifolds, erroneously extending conditions that hold only under zero sectional curvature, as detailed in \cite{KristalyLiLopez-AcedoNicolae2016}. Determining the geodesic convexity of a function on a manifold often presents formidable challenges, even for seemingly straightforward cases involving simple functions on manifolds of constant nonzero curvature, such as quadratic functions.

The seemingly basic task of discerning whether a quadratic function on a manifold of constant curvature is geodesically convex highlights the complexity inherent in this endeavor. With the burgeoning interest in manifold optimization, it is paramount to establish a robust foundation. If the research community in manifold optimization struggles to address fundamental inquiries, such as the geodesic convexity of elementary functions on basic manifolds, it raises significant concerns about the field's credibility.
\item Several important problems and applications are related to minimizing
	non-homogeneous quadratic functions on the sphere (see
		\cite{Sorensen1997,Hager2001} 
		together with the corresponding references
		\cite{ByrdSchnabelSchultz1987,Gander1980,GolubGeneMatt1991,More1983,MoreSorensen1983,Sorensen1992,Byrd1987,CellisDennisTapia1985,HagerKrylyuk1999,Menke2012,PowelYuan1990,Tarantola1987} 
listed in their introduction). These problems are related
to trust region algorithms, regularization methods for ill-posed problems, and
seismic inversion problems, as detailed in \cite{Sorensen1997,Hager2001}. The convergence and performance of possible algorithms addressing these problems can profit from the geodesic convexity displayed by such functions.
\end{enumerate}

While our study may appear peripheral in the domain of manifold optimization, the aforementioned motivations underscore the indispensable significance of this topic.
\medskip

For clarity, we refer to the intersection of the positive orthant with the sphere as the "spherical positive orthant." The key findings of the paper are summarized as follows:

\begin{enumerate}
    \item An affine function on a spherically convex set, defined by a proper cone, is spherically convex if and only if its linear part is defined by a vector in the cone's polar.
    \item While our previous work (Ferreira \& Nemeth, 2019) demonstrated that
	    all  convex, homogeneous quadratic functions on the spherical
	    positive orthant are constant, the class of non-homogeneous quadratic functions on this domain is surprisingly vast.
    \item Contrary to the case in hyperbolic space, where any non-homogeneous
	    geodesically convex function has a geodesically convex homogeneous
	    quadratic part, this property doesn't extend to the spherical
	    nonnegative orthant. Moreover, for any symmetric matrix
	    $A\in\R^{n\times n}$, there exist infinitely many non-homogeneous
	    spherically convex quadratic functions with $A$ as the matrix of their homogeneous quadratic part. The spherical convexity of these functions depends on the size of the components of the vector $b$ defining their first-order homogeneous component. The largest component of $b$ is unique if the domain of the spherically convex function is the positive orthant, being the component with the smallest modulus, while the other components of $b$ are negative.
    \item The function described in item 3 can exhibit spherically convex behavior even with a positive largest component of $b$, as demonstrated in our work. Determining the smallest upper bounds of the components of $b$ for spherically convex behavior poses a challenging question, possibly involving unknown relationships between these components as part of necessary and sufficient conditions for spherical convexity.
\end{enumerate}

The paper is structured as follows:
\begin{itemize}
	\item[-] In Section 2, we establish the general terminology used throughout the paper.
	\item[-] In Section 3, we provide a brief overview of the main concepts related to spherical convexity and revisit a crucial result concerning the spherical convexity of smooth functions, which plays a paramount role in our study.
	\item[-] Section 4 presents necessary, sufficient, and equivalent conditions for the spherical convexity of non-homogeneous quadratic functions. We specifically emphasize functions defined on the spherically convex set of unit vectors with positive coordinates, as well as functions corresponding to special classes of matrices—namely, positive matrices, Z-matrices, and diagonal matrices.
	\item[-] Finally, in Section 5, we offer concluding remarks and propose challenging questions for future exploration.
\end{itemize}

\section{Basic terminology}

Let $n\ge 3$ be an integer and $\R^n\equiv\R^{n\times 1}$ be the vector space whose elements are column vectors of length $n$ 
and real entries. 

Denote by $I\in\R^{n\times n}$ the identity matrix, that is, the matrix with elements $I_{ij}=\delta_{ij}$, where $\delta$ is the Kronecker
symbol.

For any positive integer $k$ denote $[k]=\{1,\dots,k\}$. For all $i\in [n]$ denote by 
$e^i$ the column vector with all entries $0$ except 
the $i$-th entry which is $1$. The basis $\{e^1,\dots,e^n\}$ of $\R^n$ is called
the standard basis. Denote $\vo=e^1+\dots+e^n$. For any vector $z\in\R^n$ let 
$z_i=z\tp e^i$, for all $i\in [n]$. 

For any vector $d\in\R^n$ 
and symmetric matrix $A\in\R^{n\times n}$ denote 
$\diag(d):=\sum_{i=1}^nd_ie^i(e^i)\tp\in\R^{n\times n}$, 
$\diag(A)=\sum_{i=1}^n\lng Ae^i,e^i\rng e^i\in\R^n$, $\diag^2(A)=\diag(\diag(A))\in\R^{n\times n}$ and
$a_{ij}=\lng Ae^i,e^j\rng$ for all $i,j\in [n]$. Let $I:=\sum_{i=1}^n(e^i)(e^i)\tp$.

For any 
$x=(x_1,\dots,x_n)\tp$ and $y=(y_1,\dots,y_n)\tp$ in $\R^n$ denote $\lng x,y\rng:=x\tp y$ the canonical inner product. 

For any number $\alpha\in\R$ denote $\alpha_+=\max(\alpha,0)$, $\alpha_-=\max(-\alpha,0)$;
for any vector $v=v_1e^i+\dots+v_ne^n=(v_1,\dots,v_n)\tp\in\R^n$ denote  
\[v_+=\lf(v_1^+,\dots,v_n^+\rg)\tp\in\R^n_+,\quad v_-=\lf(v_1^-,\dots,v_n^-\rg)\tp\in\R^n_+;\] and for 
any matrix 
$B\in\R^{k\times l}$ denote $B=(b_{ij})_{i\in [k],j\in [l]}$. 

Let \[\Sp=\lf\{x\in\R^n:\|x\|=1\rg\}\] 
be the unit sphere, which we will shortly call just ``sphere''. 

With a slight abuse of notation we will use the same notation for the linear mappings of $\R^n$ and their matrices with respect
to the standard basis.

A set $K\subseteq\R^n$ is called \emph{cone} if $\lambda x\in K$, for all $x\in K$ and
all $\lambda>0$. A cone $K\subseteq\R^n$ is called \emph{convex cone} if is a
if $x+y\in K$ for any $x,y\in K$. It follows that a cone $K\subseteq\R^n$ is a
convex cone if and only if it is a convex set. A cone $K\subseteq\R^n$ is called
\emph{pointed cone} if there is no $x\in K\setminus\{0\}$ such that $-x\in K$. A cone 
$K\subseteq\R^n$ is called \emph{closed cone} if it is a closed set. A cone
$K\subseteq\R^n$ is called \emph{proper} if it is a pointed closed convex cone
with nonempty interior. The \emph{dual} $K^\perp$ of a cone $K\subseteq\R^n$ is defined by 
\[K^*=\{x\in\R^n:\lng x,y\rng\ge 0,\,\forall y\in K\}.\] The \emph{polar}
$K^\perp$ of $K$ is defined by \[K^\perp=\{x\in\R^n:\lng x,y\rng\le 0,\,\forall y\in K\}.\]

\section{Preliminaries on the sphere}

The terminology and results about the sphere follow the ones in
\cite{FerreiraIusemNemeth2013,FerreiraIusemNemeth2014}. 
Therefore, we will 
broadly describe the main concepts of spherical convexity and only repeat the
formulas explicitly used later. 

The geodesics of the sphere $\Sp$ are circles which are the
intersection of two dimensional subspaces of $\R^n$ with the sphere. 

The distance
between any two points on the sphere is the length of the smallest geodesic arc
joining the two points. The diameter of a subset of the sphere is the supremum of 
distances between all pair of points belonging to the sphere. 

Any two
points on the sphere of distance less than $\pi$ can be joined by two geodesic
arcs and two diametrically opposed points by infinitely many. The smallest
length geodesic arc joining two points is called \emph{minimal}. 

A set
$\mathcal{C}\subseteq\Sp$ is called spherically convex if for any two points of
$\mathcal{C}$ all minimal geodesic arcs joining the two points are contained in
$\mathcal{C}$. A spherically convex set is called proper, if it is nonempty and
it is not the whole sphere. This definition implicitly yields that the diameter of any
proper spherically convex set is less than $\pi$ and thus the minimal geodesic arc joining any 
two points of the set is unique. 
The proper spherically convex sets are intersections of the sphere with
pointed convex cones of $\R^n$.   

A real valued function defined on a spherically
convex set is called \emph{spherically convex} if its composition with any
minimal geodesic segment belonging to the set is a convex function defined on an
interval. This definition implies
that the only spherically convex functions on the whole sphere are the constant
ones. 

Denote by $D$ and $D^2$, the Euclidean gradient and Euclidean Hessian,
respectively. The following proposition, proved in \cite{FerreiraNemeth2019}, will be paramount
for our investigations.

\begin{proposition}\label{prop:FerreiraNemeth2019} 
	Let $K\subseteq R^n$ be a proper cone, 
	$\mathcal{C}=\Sp\cap\inte(K)$ and $f:C\to\R$ a 
	smooth function. Then, the following statements are equivalent:
	\begin{enumerate}[(i)]
		\item $f$ is spherically convex; 
		\item $\lng Df(x)-Df(y),x-y\rng + (\lng x,y\rng-1)[\lng Df(x),x\rng
			+\lng Df(y),y\rng]\ge 0$, for all $x,y\in\mathcal{C}$;
		\item $\lng D^2f(y)x,x\rng-\lng Df(y),y\rng\ge 0$, for all $y\in\mathcal{C}$,
			$x\in\Sp$ with $\lng x,y\rng=0$.
	\end{enumerate}
\end{proposition}

\section{Main results}
Throughout this paper we restrict ourselves to proper cones $K\subseteq\R^n$,
that is to pointed closed convex cones with nonempty interior. Let $c\in\R$, 
$A\in\R^{n\times n}$ be a 
symmetric matrix and $b\in\R^n$. From now on $f$ will always denote a function 
$f:\Sp\cap\inte(K)\to\mathbb R$ defined by \[f(x)=\lng Ax,x\rng+\lng b,x\rng+c.\] Sometimes, to emphasize
the dependence of $f$ on $A$, $b$, $c$, we will use the more specific notation $f=f_{A,b,c}$.
Denote $\lambda_{\min}(A)$ and $\lambda_{\max}(A)$ the smallest and largest eigenvalues of $A$, 
respectively. 

The following Proposition is a direct consequence of Proposition
\ref{prop:FerreiraNemeth2019}.

\begin{proposition}
	The following statements are equivalent:
	\begin{enumerate}[(i)]
		\item $f$ is spherically convex; 
		\item 
			\begin{equation}\label{eq:foc}
				\lng Au,u\rng-\lng Av,v\rng\ge \f12\lng b,v\rng,
			\end{equation}
			for all $u\in\Sp$ and all $v\in\Sp\cap K$ with $\lng u,v\rng=0$.
		\item 
			\begin{equation}\label{eq:soc}
				4\lng Ax,y\rng\le 2\lng x,y\rng\lf(\lng Ax,x\rng
				+\lng Ay,y\rng\rg)+(\lng x,y\rng-1)\lng b,x+y\rng,
			\end{equation}
			for all $x,y\in\Sp\cap K$. 
	\end{enumerate}
\end{proposition}

The following lemma demonstrates that the spherical convexity of \(f\) remains unchanged if any constant multiple of the 
identity matrix is added to \(A\).
  
\begin{lemma}\label{lem:lam}
	Let $K$ be a proper cone, $A=A\tp\in\R^{n\times n}$, $b\in\R^n$ and $c,\lambda\in\R$. Then, 
	$f_{A,b,c}$ is spherically convex if and only if $f_{A-\lambda I,b,c}$ is spherically convex.
\end{lemma}

\begin{proof}
	For any $u,v\in\Sp$ with $v\in K$ and $\lng u,v\rng=0$ denote 
	\[E(A,b,c,u,v):=2\lng Au,u\rng-2\lng Av,v\rng-\lng b,v\rng.\] Then, 
	according to \eqref{eq:foc}, $f_{A,b,c}$ is spherically convex if and only if
	$E(A,b,c,u,v)\ge 0$ for all $u,v\in\Sp$ with $v\in K$ and $\lng u,v\rng=0$.
	It is easy to check that $E(A-\lambda I,b,c,u,v)=E(A,b,c,u,v)$. Hence, the result readily 
	follows. 
\end{proof}

The following theorem lists necessary conditions, sufficient conditions, and
equivalent conditions for $f$ to be spherically convex. The necessary conditions
(ii)-(v) can serve as negative certificates for the spherical convexity of $f$.
If any of these conditions do not hold, then $f$ is not spherically convex. The
sufficient condition (vi) provides a positive certificate for the spherical
convexity of $f$. If this condition holds, then $f$ is spherically convex.
Moreover, this condition demonstrates that for any symmetric matrix 
$A\in\mathbb{R}^{n\times n}$, there are infinitely many $b$ such that $f$ is
spherically convex, making the class of spherically convex non-homogeneous quadratic functions 
much larger than the class of (Euclidean) convex non-homogeneous quadratic functions. Note that
the vector $b$ has no influence on the (Euclidean) convexity of $f$, but it is
crucial for its spherical convexity. Even more interestingly, if $K=\R^n_+$, then this 
sharply contrasts with the 
fact that the class of spherically convex homogeneous functions is trivial, that is, the 
function $f$ is spherically convex for $b=0$ if and only if it is constant
\cite{FerreiraNemeth2019}.

\begin{theorem}\label{th:lt} 
	The following statements hold.
	\begin{enumerate}[(i)]
		\item Suppose that $A=0$. Then, $f$ is spherically convex if and only if $b\in K^\perp$.
		\item Suppose that $K^*\subseteq K$ and
			$f$ is sperically convex. Then, 
			\begin{equation}\label{eq:mix}
				\lng Ax,y\rng\le-\f{\sqrt 2}8\lng b,x+y\rng,
			\end{equation}
			for all $x,y\in\Sp\cap K$ with $\lng x,y\rng=0$. In 
			particular, if $K=\R^n_+$, then 
			\begin{equation}\label{eq:mixp}
				a_{ij}\le-\f{\sqrt2}8(b_i+b_j),
			\end{equation}
			for all $i,j\in [n]$ with $i\ne j$.
		\item\label{lc} Suppose that $f$ is spherically convex. Then, $\lng
			b,x+y\rng\le 0$, for all $x,y\in\Sp\cap K$ with $\lng x,y\rng=0$. In
			particular, if $\R^n_+\subseteq K$ and $f$ is spherically convex, then 
			\begin{equation}\label{eq:sij}
				b_i+b_j\le 0,
			\end{equation} 
			for all $i,j\in [n]$ with $i\ne j$, and
			\begin{equation}\label{eq:a}
				\min\lf\{a_{ii}:i\in [n]\rg\}\ge\max\lf\{b_j+a_{jj}:j\in [n]
				\setminus\argmin\lf\{a_{ii}:i\in [n]\rg\}\rg\}.
			\end{equation}
		\item Suppose that $f$ is spherically convex. Then, 
			\begin{equation}\label{eq:su}
				\lng b,x+y\rng\le-4\sqrt2\lng Ax,y\rng^+,	
			\end{equation} 
			for all $x,y\in\Sp\cap K$ with $\lng x,y\rng=0$.
			In particular, if $K=\R^n_+$, then 
			\begin{equation}\label{eq:sup}
				b_i+b_j\le -4\sqrt2a_{ij}^+,
			\end{equation}
			for all $i,j\in [n]$ with $i\ne j$.
		\item If $K=\R^n_+$, $A$ has positive entries and $f$ is spherically 
			convex, then 
			\begin{equation}\label{eq:bminus}
				\|b_-\|\ge2\lf[\lambda_{\max}(A)-\lambda_{\min}(A)\rg].
			\end{equation}
		\item If $K\subseteq\R^n_+$ and 
			\[b_i\le 2\sqrt{n}\lf[\lambda_{\min}(A)-\lambda_{\max}(A)\rg],\] for all 
			$i\in [n]$, then $f$ is spherically convex.  
	\end{enumerate}
\end{theorem}

\begin{proof} 
	$\,$

	\begin{enumerate}[(i)]
		\item Suppose that $A=0$. Then, by using inequality \eqref{eq:foc}, it follows that
			$f$ is spherically convex if and only if $\lng b,v\rng\le 0$ for all 
			$v\in\Sp\cap K$. Therefore, $f$ is spherically convex if and only if 
			$b\in K^\perp$.
		\item Inequality \eqref{eq:mix} is a consequence of inequality 
			\eqref{eq:foc} with $u=\lf(\sqrt2/2\rg)(x-y)$ and
			$v=\lf(\sqrt2/2\rg)(x+y)$ because $v\in K$ and it can be easily
			verified that $\|u\|=\|v\|=1$, $\lng u,v\rng=0$.  
			In particular, if $K=\R^n_+$, then inequality \eqref{eq:mixp} follows from
			\eqref{eq:mix} by taking $x=e^i$ and $y=e^j$.
		\item Suppose that $f$ is spherically convex and let any $x,y\in
			\Sp\cap K$ with $\lng x,y\rng=0$. Then, the first
			inequality of the statement follows by adding the inequalities \[\lng
			Ax,x\rng-\lng Ay,y\rng\ge\f 12\lng b,y\rng,\quad\lng
			Ay,y\rng-\lng Ax,x\rng\ge\f 12\lng b,x\rng,\] which are
			consequences of inequality \eqref{eq:foc}.

			The next inequality of the statement follows by taking any
			$i,j\in [n]$ with $i\ne j$ and $x=e^i$, $y=e^j$.
			On the other hand, inequality \eqref{eq:foc} implies 
 			\begin{equation}\label{eq:ij}
				2(a_{ii}-a_{jj})\ge b_j. 
			\end{equation}
			Inequality \eqref{eq:a} is a straightforward 
			consequence of inequality \eqref{eq:ij}.
		\item It straightforwardly follows from inequalities
			\eqref{eq:mix}, \eqref{eq:mixp}, \eqref{eq:su} and \eqref{eq:sij}. 
		\item By using the Cauchy inequality we have
			\begin{equation}\label{eq:pf}
				2\lf[\lng Au,u\rng-\lng Av,v\rng\rg]\ge\lng b,v\rng
				=\lng b_+-b_-,v\rng\ge-\lng b_-,v\rng\ge-\|b_-\|.
			\end{equation} 
			Let $u,v$ be unit eigenvectors of $A$ corresponding 
			to $\lambda_{\min}(A),\lambda_{\max}(A)$, respectively. From the Perron-Frobenius
		theorem \cite{Perron1907,Frobenius1912} it follows that we can choose $v\in\mathbb\R^n_+=K$, hence our choice of 
			$u,v$ is feasible. Then, \eqref{eq:pf} implies \eqref{eq:bminus}.  
		\item Since 
			$v\in S\cap K$, it follows 
			that there exists an $i_0\in [n]$ with $v_{i_0}\ge 1/\sqrt{n}$. 
			Indeed, otherwise 
			$0\le v_i<1/\sqrt{n}$, for all $i\in [n]$ implies $\|v\|<1$, which 
			contradicts 
			$v\in\Sp$. We have 
			\[-b_i\ge 2\sqrt{n}\lf[\lambda_{\max}(A)-\lambda_{\min}(A)\rg]\ge 0,\] 
			for all 
			$i\in [n]$, which implies 
			\[\f12\lng-b,v\rng\ge\f12(-b_{i_0})v_{i_0}\ge\lambda_{\max}(A)
			-\lambda_{\min}(A)\ge\lng Av,v\rng-\lng Au,u\rng,\] or equivalently 
			\[\lng Au,u\rng-\lng Av,v\rng\ge\f12\lng b,v\rng.\] Hence, inequality 
			\eqref{eq:foc} implies that $f$ is spherically convex. 
	\end{enumerate}
\end{proof}

The next theorem presents a characterization of spherically convex functions. Although this characterization depends on a single vector $x\in\Sp\cap\inte(K)$, it also involves an implicit 
term that resembles a scalar derivative, similar to those defined in \cite{IsacNemeth2008} (see
\cite{Nemeth1992} for the earliest definition of the scalar derivative).

\begin{theorem}\label{th:li}
	The function $f$ is spherically convex, if and only if 
	\begin{eqnarray*}
		\f12\lng b,x\rng\le\liminf_{\mathclap{\substack{y\to x\\\|y\|=1\\y\in K}}}
		\f{\lng Ay-Ax,y-x\rng}{\|y-x\|^2}-\lng Ax,x\rng,
	\end{eqnarray*}
	for all $x\in\Sp\cap\inte(K)$.
\end{theorem}

\begin{proof}
	Suppose that $f$ is spherically convex. Let $x\in\Sp\cap\inte(K)$ and $y\in\Sp\cap K$, such that $y\ne x$. Then, $1-\lng x,y\rng>0$. 
	Hence, according to \eqref{eq:soc} we have 
	\begin{equation}\label{wm1}
		\lng b,x+y\rng\le\f{4\lng x,y\rng\lf(\lng Ax,x\rng+\lng Ay,y\rng\rg)
		-8\lng Ax,y\rng}{2-2\lng x,y\rng}.
	\end{equation}
	Let $t=\|y-x\|$ and $u=(1/t)(y-x)$. Then, $y=x+tu$ and $\|y\|=\|x\|=\|u\|=1$ imply
	\begin{eqnarray*}
		2-2\lng x,y\rng=2-2\lng x,x+tu\rng=-2t\lng x,u\rng=-\|x+tu\|^2+\|x\|^2+t^2\|u\|^2=
		t^2.
	\end{eqnarray*}
	Hence, \eqref{wm1} becomes 
	\begin{eqnarray*}
		2\lng b,x\rng\le
		\f{\lf(4-2t^2\|u\|^2\rg)\lf(2\lng Ax,x\rng+2t\lng Ax,u\rng+t^2\lng Au,u\rng\rg)
		-8\lng Ax,x\rng-8t\lng Ax,u\rng}{t^2}\\
		=4\lng Au,u\rng-4\lng Ax,x\rng-4t\lng Ax,u\rng-2t^2\lng Au,u\rng-t\lng b,u\rng,
	\end{eqnarray*}
	which implies
	\begin{eqnarray*}
		2\lng b,x\rng\le\liminf_{\mathclap{\substack{y\to x\\\|y\|=1\\y\in K}}}
		\lf(4\lng Au,u\rng-4\lng Ax,x\rng-4t\lng Ax,u\rng-2t^2\lng Au,u\rng-t\lng b,u\rng\rg).
	\end{eqnarray*}
	Hence,
	\begin{eqnarray*}
		\f12\lng b,x\rng\le\liminf_{\mathclap{\substack{y\to x\\\|y\|=1\\y\in K}}}
		\f{\lng Ay-Ax,y-x\rng}{\|y-x\|^2}-\lng Ax,x\rng.
	\end{eqnarray*}
	Conversely, suppose that 
	\begin{eqnarray}\label{eq:li2}
		\f12\lng b,x\rng\le\liminf_{\mathclap{\substack{y\to x\\\|y\|=1\\y\in K}}}
		\f{\lng Ay-Ax,y-x\rng}{\|y-x\|^2}-\lng Ax,x\rng,
	\end{eqnarray}
	for all $x\in\Sp\cap\inte(K)$. Let $y=(\cos t)x+(\sin t)v$, where $t\in
	(0,2\pi)$ and $v\in\Sp\cap K$, $\lng v,x\rng=0$. Then, $y\in\Sp\cap K$ and $y\ne x$. 
	We have
	\begin{equation*}
		\f{\lng Ay-Ax,y-x\rng}{\|y-x\|^2}
		=\f{(\cos t-1)^2\lng Ax,x\rng+(\sin(2t)-2\sin t)\lng
		Av,x\rng+\sin^2 t\lng Av,v\rng}{2-2\cos t}
	\end{equation*}
	Tending with $t$ to zero and using the L'Hopital rule twice we obtain
	\begin{equation*}
		\liminf_{\mathclap{\substack{y\to x\\\|y\|=1\\y\in K}}}
		\f{\lng Ay-Ax,y-x\rng}{\|y-x\|^2}\le\lng Av,v\rng.	
	\end{equation*}
	Hence, \eqref{eq:li2} implies 
	\begin{equation*}
		\f12\lng b,x\rng\le\lng Av,v\rng-\lng Ax,x\rng
	\end{equation*}
	for all $x\in\Sp\cap \inte(K)$ and any $v\in\Sp$ with $v\perp x$. By taking a limit and using continuity, we can assume that the same
	inequality holds for all $x\in\Sp\cap K$. Hence,
	\eqref{eq:foc} holds and therefore $f$ is spherically convex.
\end{proof}

Although the proofs of Theorem \ref{th:li} and Proposition \ref{prop:sd} suggest
that Theorem \ref{th:li} is merely a reformulation of Proposition 
\ref{prop:FerreiraNemeth2019} (at least when the matrix $A$ is positive definite), it leads to 
the following Corollary, which offers a new type of characterization of spherically convex 
non-homogeneous quadratic functions. This characterization depends only on a unit vector in 
the cone and an arbitrary unit vector that is not related to it. Previously, our 
characterizations involving two unit vectors either required that both vectors be in $K$ or that
one unit vector be in $K$ and the other be perpendicular to it.

\begin{corollary}
	The function $f$ is spherically convex if and only if 
	\begin{equation}\label{eq:bx}
		\f12\lng b,x\rng
		\le\f{\lng Av,v\rng-2\lng v,x\rng\lng Av,x\rng+(2\lng
		v,x\rng^2-1)\lng Ax,x\rng}{1-\lng v,x\rng^2},
	\end{equation}
	for all $x\in\Sp\cap K$ and all unit vectors $v$ with $v\notin\{x,-x\}$.
\end{corollary}

\begin{proof}
	Let $f$ be spherically converse. First, suppose that $x\in\Sp\cap\inte(K)$. Let $v$ be a
	unit vector such
	that $v\notin\{x,-x\}$. 
	and $y=(x+tv)/\|x+tv\|$, where $t>0$ is such that $\|x+tv\|\ne 0$. If $t>0$ is
	sufficiently small then $\|x+tv\|\ne 0$ and $y\in\Sp\cap K$. Moreover $y\ne x$. 
	First, let us calculate $\psi'(0)$, where $\psi(t)=\|x+tv\|$. We have\
	\begin{eqnarray*}
		\f{\psi(t)-\psi(0)}t=\f{\|x+tv\|-1}t=\f{(\|x+tv\|-1)(\|x+tv\|+1)}{t(\|x+tv\|+1}
		=\f{\lng x+tv,x+tv\rng-1}{t(\|x+tv\|+1}\\
			=\f{2\lng x,v\rng+t\|v\|^2}{\|x+tv\|+1}
	\end{eqnarray*}
	Letting $t$ tend to zero in the above formula, we obtain
	\begin{equation*}
		\f{d}{dt}|_{t=0}\|x+tv\|=\lng v,x\rng.
	\end{equation*}
	Multiplying the numerator and 
	denominator of $\lng Ay-Ax,y-x\rng/\|y-x\|^2$ by $\|x+tv\|^2$, we obtain
	\begin{eqnarray*}
		\f{\lng Ay-Ax,y-x\rng}{\|y-x\|^2}
		=\f{\lng Ax+tAv-\|x+tv\|Ax,x+tv-\|x+tv\|x\rng}{\lng x+tv-\|x+tv\|x,x+tv-\|x+tv\|x\rng}
	\end{eqnarray*}
	Tending with $t$ to zero and using the L'Hopital rule twice, we obtain
	\begin{eqnarray*}
		\liminf_{\mathclap{\substack{y\to x\\\|y\|=1\\y\in K}}}
		\f{\lng Ay-Ax,y-x\rng}{\|y-x\|^2}\le\f{\lng Av,v\rng-2\lng v,x\rng\lng Av,x\rng
		+\lng v,x\rng^2\lng Ax,x\rng}
		{1-\lng v,x\rng^2}.\\
	\end{eqnarray*}
	Hence, Theorem \ref{th:li} implies \eqref{eq:bx}. If $v\in\{x,-x\}$, 
	then let $v^n$ be unit vectors such that $v^n\to v$ if $n\to+\infty$. Writing inequality 
	\eqref{eq:bx} with $v^n$ replacing $v$ and tending with $n$ to infinity, we obtain inequality
	\eqref{eq:bx}. If $x$ is on the the intersection of $\Sp$ 
	with the boundary of $K$, then there
	exists $x^k\in\Sp\cap\inte(K)$ such that $\lim_{k\to\infty}x^k=x$. Hence,
	\begin{equation*}
		\f12\lng b,x^k\rng\le\f{\lng Av,v\rng-2\lng v,x^k\rng\lng Av,x^k\rng
		+(2\lng v,x^k\rng^2-1)\lng Ax^k,x^k\rng}{1-\lng v,x^k\rng^2},
	\end{equation*}
	Tending with $k$ to infinity 
	in the last inequality we obtain \eqref{eq:bx}. 
	Conversely suppose that 
	\begin{equation*}
		\f12\lng b,x\rng
		\le\f{\lng Av,v\rng-2\lng v,x\rng\lng Av,x\rng+(2\lng
		v,x\rng^2-1)\lng Ax,x\rng}{1-\lng v,x\rng^2},
	\end{equation*}
	for all $x\in\Sp\cap K$ and all unit vectors $v$ with $v\notin\{x,-x\}$.
	Let $v\in\Sp$ such that $\lng v,x\rng=0$. Then, inequality \eqref{eq:foc}
	holds and therefore $f$ is spherically convex.
\end{proof}

Without entering in details, we remark that the inferior limit in Theorem \ref{th:li} is a 
variant of the scalar derivative defined by S. Z. N\'emeth in \cite{Nemeth1992}
(for other variants see \cite{IsacNemeth2008}). In the next 
proposition we show how to determine this inferior limit in the case when $A$ is positive
definite. Note that, by Lemma \ref{lem:lam}, the positive definiteness of $A$ is not a restriction on the
spherical convexity of $f=f_{A,b,c}$ because the spherical convexity of $f$ is
equivalent to the spherical convexity of $f_{A+\alpha I,b,c}$ and we can take
$\alpha$ sufficiently large for $A+\alpha I$ to be spherically convex.

\begin{proposition}\label{prop:sd}
	Suppose that $A$ is positive definite and let $x\in\Sp\cap\inte(K)$. Then, 
	\begin{equation}\label{eq:sd}
		\liminf_{\mathclap{\substack{y\to x\\\|y\|=1\\y\in K}}}
		\f{\lng
		Ay-Ax,y-x\rng}{\|y-x\|^2}=\min_{\substack{u\in\Sp\\\lng
		u,x\rng=0}}\lng Au,u\rng=\lambda_{\min}(P_xAP_x+\lambda Q_x), 	
	\end{equation}
	where $P_x=I-xx\tp$, $Q_x=xx\tp$ and $\lambda=\lng Ar,r\rng$, for any
	fixed $r\in\Sp$ with $\lng r,x\rng=0$.
\end{proposition}

\begin{proof}
	First we justify the second inequality in \eqref{eq:sd}. Let 
	$x\prp=\{u\in\R^n\,:\,\lng u,x\rng=0\}$. Please note the following
	properties of $P_x$ and $Q_x$:
	\begin{equation}\label{eq:pxqx}
		\begin{array}{r}
			P_x+Q_x=I,\quad P_x^2=P_x=P_x\tp,\quad Q_x^2=Q_x=Q_x\tp,\quad
			P_xQ_x=0,\quad P_xx=0,\quad Q_xx=x,\\
			w\in\R^n\implies P_x w=w-\lng w,x\rng x\in x\prp,\quad
			z\in x\prp\implies P_x z=z
		\end{array}
	\end{equation}
	Formula \eqref{eq:pxqx}$_7$ implies that $P_xAP_xx^\perp\subseteq x\prp$.
	Let any $z\in x\prp$. Since $A$ is positive definite, $A\inv$ is also positive definite.
	Therefore, $\lng A^{-1}x,x\rng>0$ and the vector 
	\[v:=A\inv z-\f{\lng A\inv z,x\rng}{\lng A\inv x,x\rng}A\inv x\in x\prp\]
	is well defined. By using \eqref{eq:pxqx}$_{8,5}$, we obtain
	$P_xAP_xx\prp\ni P_xAP_xv=P_xAv=z$. Hence, $x\prp\subseteq
	P_xAP_xx^\perp$. In conclusion $P_xAP_xx^\perp=x\prp$ and therefore $x\prp$ is an 
	invariant hyperplane of $P_xAP_x$. With a slight abuse of notation, we
	will identify the linear operators of $\R^n$ with their matrices with
	respect to the standard canonical basis. Then, we have 
	\[\lambda_{\min}(P_xAP_x|_{x\prp})=\min_{\substack{u\in\Sp\\\lng
	u,x\rng=0}}\lng Au,u\rng.\] To justify the second inequality in
	\eqref{eq:sd}, it remains to show that 
	\begin{equation}\label{eq:lieq}
		\lambda_{\min}(P_xAP_x+\lambda
		Q_x)=\lambda_{\min}(P_xAP_x|_{x\prp}).
	\end{equation}
	Let $w\in\R^n\setminus\{0\}$ be an eigenvector of $P_xAP_x+\lambda Q_x$
	with $P_xw\ne 0$ and corresponding eigenvalue $\mu$. Multiplying 
	the left and right hand sides of 
	\begin{equation}\label{eq:beig}
		(P_xAP_x+\lambda Q_x)w=\mu w
	\end{equation}
	by $P_x$
	and using \eqref{eq:pxqx}$_{2,4,7}$, we obtain $P_xAP_xP_xw=\mu P_xw$,
	which implies that $\mu$ is an eigenvalue of $P_xAP_x|_{x\prp}$ with
	corresponding eigenvector $P_xw\in x\prp$.
	Hence, 
	\begin{equation}\label{eq:ieig}
		\mu\ge\lambda_{\min}(P_xAP_x|_{x\prp}).
	\end{equation}
	Next, let
	$w\in\R^n\setminus\{0\}$ be an eigenvector of $P_xAP_x+\lambda Q_x$ with
	$P_x w=0$ and the corresponding eigenvalue $\mu$. Since $P_x w=0$, it
	follows form \eqref{eq:pxqx}$_7$ that $w=\lng w,x\rng x$. Hence,
	equations \eqref{eq:pxqx}$_{5,6}$ and \eqref{eq:beig} imply 
	\begin{equation}\label{eq:canc}
		\mu\lng w,x\rng=(P_xAP_x+\lambda Q_x)\lng w,x\rng=\lambda\lng
		w,x\rng.
	\end{equation}
	Note that $\lng w,x\rng\ne 0$ because otherwise $P_xw=w-\lng w,x\rng
	x=w\ne 0$ contradicts our assumption $P_x w=0$. Thus, \eqref{eq:canc},
	$v\in S\cap x\prp$ and \eqref{eq:pxqx}$_{2,8}$ implies 
	\[\mu=\lambda=\lng Ar,r\rng=\lng AP_xr,P_xr\rng=\lng P_xAP_xr,r\rng\ge
	\lambda_{\min}(P_xAP_x|_{x\prp}).\] Hence, 
	\begin{equation}\label{eq:ieig2}
		\mu\ge\lambda_{\min}(P_xAP_x|_{x\prp}).
	\end{equation}
	Inequalities \eqref{eq:ieig} and \eqref{eq:ieig2} imply that
	\begin{equation}\label{eq:2d}
		\lambda_{\min}(P_xAP_x+\lambda
		Q_x)\ge\lambda_{\min}(P_xAP_x|_{x\prp}).
	\end{equation}
	Conversely, let $a\in x^{\prp}$ be an eigenvector of $P_xAP_x|_{x\prp}$
	with corresponding eigenvalue $\mu$. Then $P_xAP_xa=\mu a$ and
	\eqref{eq:pxqx}$_{8,4}$ imply 
	\[(P_xAP_x+\lambda Q_x)a=(P_xAP_x)a+\lambda Q_x a=\mu a+\lambda
	Q_xP_xa=\mu a,\] which implies that $\mu$ is an eigenvalue of
	$P_xAP_x+\lambda Q_x$ and thus $\mu\ge\lambda_{\min}(P_xAP_x+\lambda
	Q_x)$. In particular,
	\begin{equation}\label{eq:2c}
		\lambda_{\min}(P_xAP_x|_{x\prp})\ge\lambda_{\min}(P_xAP_x+\lambda
		Q_x). 
	\end{equation}
	Inequalities \eqref{eq:2d} and \eqref{eq:2c} imply inequality
	\eqref{eq:lieq} and in conclusion \eqref{eq:lieq}$_2$ holds.
	\bigskip

	To show the first equality in \eqref{eq:sd} please note that we have
	already shown in the proof of Theorem \ref{th:li} that
	\begin{equation*}
		\liminf_{\mathclap{\substack{y\to x\\\|y\|=1\\y\in K}}}
		\f{\lng Ay-Ax,y-x\rng}{\|y-x\|^2}\le\lng Au,u\rng,	
	\end{equation*}
	for all $u\in\Sp$ with $u\perp x$, which implies
	\begin{equation*}
		\liminf_{\mathclap{\substack{y\to x\\\|y\|=1\\y\in K}}}
		\f{\lng
		Ay-Ax,y-x\rng}{\|y-x\|^2}\le\min_{\substack{u\in\Sp\\\lng
		u,x\rng=0}}\lng Au,u\rng. 	
	\end{equation*}
	It remains to show that
	\begin{equation}\label{eq:sd1ge}
		\liminf_{\mathclap{\substack{y\to x\\\|y\|=1\\y\in K}}}
		\f{\lng
		Ay-Ax,y-x\rng}{\|y-x\|^2}\ge\min_{\substack{u\in\Sp\\\lng
		u,x\rng=0}}\lng Au,u\rng. 	
	\end{equation}
	Let any $\varepsilon>0$ and $y^k\in\Sp\cap K$ be a sequence such that $y^n\ne x$ for 
	all positive integer $k$, $\lim_{k\to\infty}y^k=x$ and the sequence 
	\begin{equation}\label{eq:ytoz}
		\f{\lng Ay^k-Ax,y^k-x\rng}{\|y^k-x\|^2}
		=\lng A\lf(\f{y^k-x}{\|y^k-x\|}\rg),\f{y^k-x}{\|y^k-x\|}\rng=\lng
		Az^k,z^k\rng
	\end{equation} 
	is convergent to
	\begin{equation}\label{eq:eps}
		\liminf_{\mathclap{\substack{y\to x\\\|y\|=1\\y\in K}}}
		\f{\lng Ay-Ax,y-x\rng}{\|y-x\|^2}-\varepsilon,	
	\end{equation}
	where $z^k:=(y^k-x)/\|y^k-x\|$.
	Since $z^k\in\Sp$ and $\Sp$ is compact, it follows that
	there exists a convergent subsequence $z^{k_\ell}$ of $z^k$ such that
	\begin{equation}\label{eq:v}
		v:=\lim_{\ell\to\infty}z^{k_\ell}\in\Sp.
	\end{equation}
	We have 
	\begin{equation}\label{eq:liv}
		\lim_{\ell\to\infty}\lng Az^{k_\ell},z^{k_\ell}\rng=\lng Av,v\rng.
	\end{equation}
	On the other hand we have 
	\begin{equation*}
		\lng y^{k_\ell}+x,z^{k_\ell}\rng=\lng
		y^{k_\ell}+x,\f{y^{k_\ell}-x}{\|y^{k_\ell}-x\|}\rng
		=\f{\|y^{k_\ell}\|^2-\|x\|^2}{\|y^{k_\ell}-x\|}=0
	\end{equation*}
	Hence, 
	\begin{equation}\label{eq:vx}
		\lng v,x\rng=\f12\lim_{\ell\tp\infty}\lng
		y^{k_\ell}+x,z^{k_\ell}\rng=0
	\end{equation}   
	Formula \eqref{eq:ytoz} with $k_\ell$ replacing $k$, \eqref{eq:eps},\eqref{eq:v}, 
	\eqref{eq:liv} and
	\eqref{eq:vx} imply
	\begin{equation*}
		\liminf_{\mathclap{\substack{y\to x\\\|y\|=1\\y\in K}}}
		\f{\lng
		Ay-Ax,y-x\rng}{\|y-x\|^2}-\varepsilon=\lng Av,v\rng
		\ge\min_{\substack{u\in\Sp\\\lng u,x\rng=0}}\lng Au,u\rng, 	
	\end{equation*}
	for all $\varepsilon>0$, which implies \eqref{eq:sd1ge} and in conclusion 
	\eqref{eq:sd1ge} holds.
\end{proof}

Proposition \ref{prop:sd} and Theorem \ref{th:li} (or equation \eqref{eq:foc}) imply the following theorem for
 \( x \in \Sp \cap \inte(K) \). To extend the theorem to \( x \in \Sp \cap K \), one needs to approach \( x \) with a 
sequence of points in \( \Sp \cap \inte(K) \) and use continuity.

\begin{theorem}
	Suppose that $A$ is positive definite. The function $f$ is spherically convex if and only if
	\begin{eqnarray*}
		\f12\lng b,x\rng\le\lambda_{\min}(P_xAP_x+\lambda Q_x)-\lng Ax,x\rng,
	\end{eqnarray*}
	for all $x\in\Sp\cap K$,
	where $P_x=I-xx\tp$, $Q_x=xx\tp$ and $\lambda=\lng Ar,r\rng$, for some 
	$r\in\Sp$ with $\lng r,x\rng=0$.
\end{theorem}

The following proposition is related to item \eqref{lc} of Theorem \ref{th:lt}.
It shows that the condition \(\langle b, u+v\rangle\leq 0\) for arbitrary perpendicular unit 
vectors with nonnegative components is no more general than the same condition
with two different vectors from the standard basis.

\begin{proposition}
	Let $K=\R^n_+$, $b\in\R^n$ and $u,v\in\Sp\cap\R^n_+$ such that $\lng u,v\rng=0$. If $b_i+b_j\le 0$, for
	all $i,j\in [n]$ with $i\ne j$, then $\lng b,u+v\rng\le 0$.
\end{proposition}

\begin{proof}
	If $b_i\le 0$, for all $i\in [n]$, then the statement is trivial. Suppose that $b$ has a 
	positive coordinate. By permutating the coordinates, we can suppose without loss of generality 
	that $b_1>0$, $b_i\le 0$, for all $i\in [n]\setminus\{1\}$, $u=\sum_{i=1}^k u_ie^i$ and
	$v=\sum_{j=k+1}^n v_je^j$. Then,
	 \begin{eqnarray*}
		 \langle b,u+v\rangle=b_1u_1+\sum_{i=2}^k b_iu_i+\sum_{j=k+1}^n b_jv_j\le
		b_1u_1+\sum_{j=k+1}^n b_jv_j\le b_1+\sum_{j=k+1}^n b_jv_j\\
		\le b_1+\max(b_{k+1},\dots,b_n)\sum_{j=k+1}^n v_j
		\le b_1+\max(b_{k+1},\dots,b_n)\sum_{j=k+1}^n v_j^2=b_1+\max(b_{k+1},\dots,b_n)\\
		\le 0.
	\end{eqnarray*}
\end{proof}

Note that, according to item \eqref{lc} of Theorem \ref{th:lt}, it is easy to prove that if \( f \) is spherically convex, 
then at most one component of \( b \) can be positive. The following proposition demonstrates that if \( f \) is spherically convex, then the components of \( b \) corresponding to the 
non-minimal elements of the diagonal of \( A \) are nonpositive. Additionally, it provides an upper bound for each 
component of \( b \) corresponding to a minimal element of the diagonal of \( A \). Some specific cases are also presented.   

\begin{proposition}\label{prop:b}
	Let $K=\R^n_+$ and any $m\in\argmin\{a_{\ell\ell}:\ell\in [n]\}$. If $f$ is spherically convex, then $b_i\le 0$, 
	for all $i\in [n]\setminus\{m\}$ and  
	\begin{equation}\label{eq:bm}
		b_m\le-\max\lf\{b_i+4a_{mi}^+:i\in [n]\setminus\{m\}\rg\}.
	\end{equation} 
	In particular, if $a_{11}=\dots=a_{nn}$ or $4a_{mi}^+\ge -b_i$, for some 
	$m\in\argmin\{a_{\ell\ell}:\ell\in [n]\}$ and some $i\in [n]\setminus\{m\}$, then 
	$b_i\le 0$ for all $i\in [n]$. 
\end{proposition}

\begin{proof}
	The inequality $b_i\le 0$ for any $i\in [n]\setminus\{m\}$ follows from inequality \eqref{eq:foc}
	with $u=e^j$ and $v=e^i$, for any $i,j\in [n]$ with $i\ne j$. Inequality \eqref{eq:bm} 
	follows from inequality \eqref{eq:sup}.
\end{proof}

The following proposition shows that for a spherically convex \( f \), if we delete the row and column of \( A \) 
corresponding to any of its minimal diagonal elements and impose that the eigenvalue of the resulting matrix is less than
the corresponding diagonal element, then all components of \( b \) are nonpositive.

\begin{proposition}
	Let $K=\R^n_+$, and for any $k\in [n]$ denote 
	$A_{-k}=(a_{ij})_{i,j\in [n]\setminus\{k\}}$. If 
	$\lambda_{\min}(A_{-m})\le a_{mm}$, for all 
	$m\in\argmin\{a_{\ell\ell}:\ell\in [n]\}$ and  
	and $f$ is spherically convex, then $b_i\le 0$, for all $i\in [n]$. \end{proposition}

\begin{proof}
	Let any $m\in\argmin\{a_{\ell\ell}:\ell\in [n]\}$. Take 
	$u=\argmin\{\lng Ax,x\rng:x\in\Sp,x\tp e^m=0\}$ and $v=e^m$. Then, 
	$\lng Au,u\rng=\lambda_{\min}(A_{-m})$. Hence, inequality \eqref{eq:foc} implies that
	\[0\ge\lambda_{\min}(A_{-m})-a_{mm}=\lng Au,u\rng-\lng Av,v\rng\ge \f12\lng b,v\rng=b^m.\]
	The inequalities $b_i\le 0$, for any $i\in [n]\setminus\{m\}$ follows from proposition
	\ref{prop:b}.
\end{proof}

The following theorem provides sufficient conditions for \( f \) to be spherically convex when \( K \) is the 
nonnegative orthant, based on the copositivity of certain matrices dependent on \( A \). Additionally, item (iii) 
demonstrates that the inequality in item (i) is implied by item (ii).

\begin{theorem}\label{th:suffd}
	Let $K=\R^n_+$. Consider the following statements
	\begin{enumerate}[(i)]
		\item\label{oc} The matrix $\diag^2(A)-A$ is copositive and 
			\(b_i\le 2[\lambda_{\min}(A)-a_{ii}],\) for all $i\in [n]$.
		\item\label{cp} The matrix $2[\lm(A)I-A]-\diag(b)$ is copositive and $b_i\le 0$, for all
			$i\in [n]$.
		\item\label{fi} The inequality \(b_i\le 2\lf[\lambda_{\min}(A)-a_{ii}\rg]\) holds for 
			all $i\in [n]$.
		\item\label{sc} The function $f$ is sperically convex.
	\end{enumerate}
	Then, we have the following implications: 
	\eqref{cp}$\implies$\eqref{fi} and \eqref{oc}$\implies$\eqref{cp}$\implies$\eqref{sc}. 
\end{theorem}

\begin{proof}
	Let $B:=\diag^2(A)-A$ and $C:=2[\lm(A)I-A]-\diag(b)$. 
	\begin{description}
		\item \eqref{cp}$\implies$\eqref{fi}: Let any $i\in [n]$.
			Since $C$ is copositive and $e^i\in K$, we have 
			$\lng Ce^i,e^i\rng=2[\lm(A)-a_{ii}]-b_i\ge 0$, which implies  
			\(b_i\le 2[\lambda_{\min}(A)-a_{ii}].\)
		\item \eqref{oc}$\implies$\eqref{cp}: Since \(b_i\le 2[\lambda_{\min}(A)-a_{ii}],\) we 
			have that $2\lm(A)I-2\diag^2(A)-\diag(b)$ is a diagonal matrix with nonnegative 
			elements, and hence copositive. Thus, we have that 
			$C=2B+2\lm(A)I-2\diag^2(A)-\diag(b)$ is copositive, because it is the sum of two
			copositive matrices. Since $e^i\in\Sp$, we also have 
			$b_i\le 2[\lambda_{\min}(A)-a_{ii}]
			=\min_{z\in\Sp}\lng Az,z\rng-\lng Ae^i,e^i\rng\le 0$.
		\item \eqref{cp}$\implies$\eqref{sc}: Let any $u,v\in\Sp$ such that $v\in K$.
			Denote by $v^2\in K^{\*}$ the vector of coordinates $(v^i)^2$ and by 
			$u^2\in K^{\*}$ the vector of coordinates $(u_i)^2$. 
			Since $v\in\Sp\cap K$, we have $v-v^2\in K=K^{\*}$.	Thus, the 
			copositivity of $C$, $-b\in K=K^{\*}$ and $v-v^2\in K^{\*}$ imply
			\begin{eqnarray*}
				2[\lng Au,u\rng-\lng Av,v\rng]=2[\lng Au,u\rng-\lm(A)]
				-2[\lng Av,v\rng-\lm(A)]\\
				=2[\lng Au,u\rng-\min_{z\in\Sp}\lng Az,z\rng]-2[\lng Av,v\rng-\lm(A)]
				\ge\lng 2[\lm(A)I-A]v,v\rng\\=\lng Cv,v\rng
				+\lng\diag(b)v,v\rng\ge\lng\diag(b)v,v\rng=\lng b,v^2\rng
				=\lng -b,v-v^2\rng+\lng b,v\rng\ge\lng b,v\rng.
			\end{eqnarray*}
			Hence, the inequality \eqref{eq:foc} holds and therefore $f$ is spherically 
			convex
	\end{description}
\end{proof}

The following corollary provides a more specific certificate for the spherical convexity of $f$ when $A$ is a $Z$-matrix. 

\begin{corollary}\label{cor:sd-1}
	Let $K=\R^n_+$. Suppose that $A$ is 
	a $Z$-matrix and \[b_i\le 2\lf[\lambda_{\min}(A)-a_{ii}\rg],\] for all
	$i\in[n]$. Then, $f$ is spherically convex.
\end{corollary}

\begin{proof}
	It follows from Theorem \ref{th:suffd} \eqref{oc}$\implies$\eqref{sc} 
	because $\diag^2(A)-A$ is a matrix with nonnegative
	elements and hence it is copositive.
\end{proof}

The next corollary is a specialization of Corollary \ref{cor:sd-1} for diagonal matrices.

\begin{corollary}\label{cor:suffd}
	Let $K=\R^n_+$, $d\in\R^n$ and $A=\diag(d)$ be a diagonal matrix. If 
	\[b_i\le 2\lf[\min(d_1,\dots,d_n)-d_i\rg]\] for all $i\in [n]$, then $f$ is spherically convex. 
\end{corollary}

\begin{proof}
	It follows from Corollary \ref{cor:sd-1} because
	$A$ is a $Z$-matrix and $\lambda_{min}(A)=\min(d_1,\dots,d_n)$.   
\end{proof}

The next theorem provides a necessary and sufficient condition for the spherical convexity of $f$ when $A$ contains at least
two minimal diagonal elements. 

\begin{theorem}\label{th:iffdiag}
	Let $K=\R^n_+$, $d\in\R^n$ and $A=\diag(d)$ be a diagonal matrix such that $
	\exists i_0,j_0\in\argmin\{d_i:i\in [n]\}$ with $i_0\ne j_0$. Then, $f$ is spherically convex 
	if and only if \[b_i\le 2\lf[\min(d_1,\dots,d_n)-d_i\rg]\] for all $i\in [n]$. 
\end{theorem}

\begin{proof}
	Suppose that $f$ is spherically convex and let any $i\in [n]$. Choose 
	any \[i_0\in\argmin\lf\{d_i:i\in [n]\rg\}\] with 
	$i_0\ne i$. Let $u=e^{i_0}$ and $v=e^i$. Then inequality \eqref{eq:foc} 
	implies $2\lf(d_{i_0}-d_i\rg)\ge b_i$. Hence, \[b_i\le 2\lf[\min(d_1,\dots,d_n)-d_i\rg].\]
	The converse follows from Corollary \ref{cor:suffd}. 
\end{proof}

The following lemma has a crucial role in the proof of Lemma \ref{lem:sq}. 

\begin{lemma}\label{lem:i}
	Let $a,u,v\in\R^n$ such that $\|a\|=\|u\|=\|v\|=1$ and $\lng u,v\rng=0$.
	Then, 
	\begin{equation*}
		\lng a,u\rng^2+\lng a,v\rng^2\le 1,
	\end{equation*}
\end{lemma}
\begin{proof}
	Let $\theta$ be the angle between $u$ and $a$, and $\varphi$ be the 
	angle between $v$ and $a$. If $a=\pm u$ or $a=\pm v$, then 
	$\lng a,u\rng^2+\lng a,v\rng^2=1$. Suppose that $a\ne\pm u$ and 
	$a\ne\pm v$. We have $u=(\cos\theta)a+(\sin\theta)b$ and 
	$u=(\cos\varphi)a+(\sin\varphi)c$, where $\|b\|=\|c\|=1$ and $\lng
	a,b\rng=\lng a,c\rng=0$. The formula \[0=\lng
	u,v\rng=\cos\theta\cos\varphi+\sin\theta\sin\varphi\lng b,c\rng\] implies
	$\cot\theta\cot\varphi=-\lng b,c\rng\in [-1,1].$ Hence
	$\cot^2\theta\cot^2\varphi\le 1$, which yields
	\[\lng a,u\rng^2+\lng a,v\rng^2=\cos^2\theta+\cos^2\varphi=2-\sin^2\theta-\sin^2\varphi
	=2-\f1{1+\cot^2\theta}-\f1{1+\cot^2\varphi}\le 1\]
\end{proof}
%


Due to Lemma \ref{lem:lam}, the following lemma is essentially not less general than Theorem \ref{th:bipos}. We present 
it here primarily because proving it is technically more convenient than directly proving Theorem \ref{th:bipos}.

\begin{lemma}\label{lem:sq}
	Let $K=\R^n_+$, $d\in\R^n_+$, $\tau$ be a permuation of $[n]$ 
	such that 
	\begin{equation}\label{eq:ch}
		0=d_{\tau(1)}<d_{\tau(2)}=\dots=d_{\tau(n)} 
	\end{equation}
	and $A=\diag(d)$. If 
	$b_{\tau(1)}\le 2d_{\tau(2)}$ and \[b_i\le -2d_{\tau(2)}\sqrt{6\sqrt3-9},\] 
	for all $i\in [n]\setminus\{\tau(1)\}$, then $f$ is spherically convex. 
\end{lemma}
\begin{proof}
	%
	Suppose that $b_{\tau(1)}\le 2d_{\tau(2)}$ and  $b_i\le -2d_{\tau(2)}\sqrt{6\sqrt3-9}$, which is 	 equivalent to $b_i\le-2d_i-\alpha$, 
	where \[\alpha:=2\lf(-1+\sqrt{6\sqrt3-9}\rg)d_{\tau(2)}\approx (0.36)d_{\tau(2)}>0,\] for all
	$i\in [n]\setminus\{\tau(1)\}$. If $b_{\tau(1)}\le 0$, then the spherical convexity of 
	$f$ follows form Corollary \ref{cor:suffd}. Suppose that $b_{\tau(1)}>0$. Then, $v\in K$,
	$b_{\tau(1)}\le 2d_{\tau(2)}$ and $b_i\le -2d_i-\alpha$, for all 
	$i\in [n]\setminus\{\tau(1)\}$ implies $b_{\tau(1)}v_{\tau(1)}\le 2d_{\tau(2)}v_{\tau(1)}$, 
	$\lng b,v\rng-b_{\tau(1)}v_{\tau(1)}\le -2\lng d,v\rng-\alpha\lng 1,v\rng+\alpha v_{\tau(1)}$. 
	By summing up the last two inequalities, we obtain
	\begin{equation}\label{eq:c1}
		\lng b,v\rng=b_{\tau(1)}v_{\tau(1)}+\lng b,v\rng-b_{\tau(1)}v_{\tau(1)}
		\le 2d_{\tau(2)}v_{\tau(1)}-2\lng d,v\rng-\alpha\lng\vo,v\rng+\alpha v_{\tau(1)}.
	\end{equation}
	We also have
	\begin{eqnarray}\label{eq:c2}
		2\lng d,v\rng+\alpha\lng\vo,v\rng-\alpha v_{\tau(1)}
		\ge\lf(2d_{\tau(2)}+\alpha\rg)\lf(\lng\vo,v\rng-v_{\tau(1)}\rg)\nonumber\\
		\ge\lf(2d_{\tau(2)}+\alpha\rg)\sqrt{\lng\vo,v^2\rng-v_{\tau(1)}^2}
		=\lf(2d_{\tau(2)}+\alpha\rg)\sqrt{\|v\|^2-v_{\tau(1)}^2}
		=\lf(2d_{\tau(2)}+\alpha\rg)\sqrt{1-v_{\tau(1)}^2}
	\end{eqnarray}
	because $v_i\ge 0$ for all $i\in [n]$. By combining \eqref{eq:c1} and \eqref{eq:c2} and 
	using the equations \eqref{eq:ch}, $d_{\tau(1)}=0$, $1=\|u\|^2=\lng 1,u^2\rng$, 
	$1=\|v\|^2=\lng 1,v^2\rng$, we get 
	\begin{eqnarray}
		2\lng d,u^2\rng-2\lng d,v^2\rng-\lng b,v\rng
		\ge 2\lng d,u^2\rng-2\lng d,v^2\rng-2d_{\tau(2)}v_{\tau(1)}+2\lng d,v\rng
		+\alpha\lng\vo,v\rng-\alpha v_{\tau(1)}\nonumber\\
		\ge 2\lng d,u^2\rng-2\lng d,v^2\rng-2d_{\tau(2)}v_{\tau(1)}
		+\lf(2d_{\tau(2)}+\alpha\rg)\sqrt{1-v_{\tau(1)}^2}\nonumber\\
		=2d_{\tau(2)}\lng\vo,u^2-u^2_{\tau(1)}e^{\tau(1)}\rng
		-2d_{\tau(2)}\lng\vo,v^2-v^2_{\tau(1)}e^{\tau(1)}\rng
		-2d_{\tau(2)}v_{\tau(1)}+\lf(2d_{\tau(2)}+\alpha\rg)\sqrt{1-v_{\tau(1)}^2}\nonumber\\
		=2d_{\tau(2)}\lng\vo,u^2\rng-2d_{\tau(2)}u^2_{\tau(1)}
		-\lf(2d_{\tau(2)}\lng\vo,v^2\rng-2d_{\tau(2)}v^2_{\tau(1)}\rg)
		-2d_{\tau(2)}v_{\tau(1)}\nonumber\\
		+\lf(2d_{\tau(2)}+\alpha\rg)\sqrt{1-v_{\tau(1)}^2}\nonumber\\
		=2d_{\tau(2)}\|u\|^2-2d_{\tau(2)}u^2_{\tau(1)}
		-\lf(2d_{\tau(2)}\|v\|^2-2d_{\tau(2)}v^2_{\tau(1)}\rg)
		-2d_{\tau(2)}v_{\tau(1)}\nonumber\\
		+\lf(2d_{\tau(2)}+\alpha\rg)\sqrt{1-v_{\tau(1)}^2}\nonumber\\
		=2d_{\tau(2)}-2d_{\tau(2)}u^2_{\tau(1)}
		-\lf(2d_{\tau(2)}-2d_{\tau(2)}v^2_{\tau(1)}\rg)
		-2d_{\tau(2)}v_{\tau(1)}+\lf(2d_{\tau(2)}+\alpha\rg)\sqrt{1-v_{\tau(1)}^2}\nonumber\\
		=2d_{\tau(2)}\lf(v^2_{\tau(1)}-u^2_{\tau(1)}-v_{\tau(1)}\rg)
		+\lf(2d_{\tau(2)}+\alpha\rg)\sqrt{1-v_{\tau(1)}^2}\label{eq:lll}
	\end{eqnarray}
	From Lemma \ref{lem:i} with $a=e^{\tau(1)}$ we have $-u_{\tau(1)}^2\ge v_{\tau(1)}^2-1$, which 
	combined with \eqref{eq:lll} imply
	\begin{eqnarray}
		2\lng d,u^2\rng-2\lng d,v^2\rng-\lng b,v\rng
		\ge 2d_{\tau(2)}\lf(2v^2_{\tau(1)}-v_{\tau(1)}-1\rg)
		+\lf(2d_{\tau(2)}+\alpha\rg)\sqrt{1-v_{\tau(1)}^2}\nonumber\\
		=2d_{\tau(2)}\lf(2x^2-x-1+\beta\sqrt{1-x^2}\rg)\ge 0,\label{eq:li}
	\end{eqnarray}
	where $x:=v_{\tau(1)}$ and $\beta:=(2d_{\tau(2)}+\alpha)/2d_{\tau(2)}=\sqrt{6\sqrt3-9}$, and
	the last inequality in \eqref{eq:li} above can be checked by showing that $\beta$ is the maximal 
	value of the function $\psi:\R\setminus\{-1,1\}\to\R$ defined by 
	\[\psi(x)=\f{-2x^2+x+1}{\sqrt{1-x^2}}\] in the
	interval $[0,1)$. Indeed, it can be easily verified that $\psi$ is concave and its unique 
	stationary point is $x^*=(\sqrt3-1)/2\in [0,1)$ (which is its unique global maximiser) and 
	$\psi(x^*)=\beta$. Hence $\beta=\max_{x\in [0,1)}\psi(x)$ and the last inequality in 
	\eqref{eq:li} readily follows.
\end{proof}

 
The next theorem demonstrates that there exist spherically convex functions $f$ such that one component of $b$ is positive.

\begin{theorem}\label{th:bipos}
	Let $K=\R^n_+$, $\tau$ be a permuation of $[n]$ 
	such that \[d_{\tau(1)}<d_{\tau(2)}=\dots=d_{\tau(n)}\] and $A=\diag(d)$. If 
	$b_{\tau(1)}\le 2\lf(d_{\tau(2)}-d_{\tau(1)}\rg)$ and 
	\[b_i\le -2\lf(d_{\tau(2)}-d_{\tau(1)}\rg)\sqrt{6\sqrt3-9},\] 
	for all $i\in [n]\setminus\{\tau(1)\}$, then $f$ is spherically convex. 
\end{theorem}

\begin{proof}
From Lemma \ref{lem:lam} $f$ is spherically convex if and only if $f_{A-d_{\tau(1)}I,b,c}$ is spherically
convex. Hence, the result readily follows. 
\end{proof}

The next proposition can be used as a further negative certificate for spherically
convex functions. Indeed, if $f$ does not satisfy the condition of the
Proposition for some indices $i\ne j$, then $f$ is not spherically convex.
Solving the corresponding minimization problem is just a matter of simple
numerical analysis. 

\begin{proposition}
	Let $K=\mathbb R^n_+$. If $f$ is sperically convex, then 
	\[\min\lf\{\lf[(a_{ii}-a_{jj})\cos(2\theta)-2a_{ij}\sin(2\theta)
	-\sin(\theta)b_i-\cos(\theta)b_j\rg]\,:\,\theta\in
	\lf[0,\f{\pi}2\rg]\rg\}\ge 0,\] for all $i,j\in [n]$ with $i\ne j$.
\end{proposition}

\begin{proof}
	Let any $\theta\in [0,\pi/2]$ and any $i,j\in [n]$ with $i\ne j$. Then,
	the result follows from \eqref{eq:foc} with
	\(u=\cos{\theta}e^i-\sin{\theta}e^j\in\Sp\) and 
	\(v=\sin(\theta)e^i+\cos(\theta)e^j\in\Sp\cap K\) after some trigonometric 
	manipulations.
\end{proof}

The next two lemmas are building blocks of Theorem \ref{th:cd}. Although in the
case $n>2$ Lemma \ref{lem:d} follows from Theorem \ref{th:iffdiag}, we give
an explicit proof below which works for all $n$. 

\begin{lemma}\label{lem:d}
	Let $K=\R^n_+$, $i\in [n]$ and $A=\pm e^i(e^i)\tp$. Then, $f$ is spherically convex if and only 
	if 
	$b_j\le -2\delta_{ij}$, for all $j\in [n]$.
\end{lemma}

\begin{proof} 
	%
	Suppose that $f$ is spherically convex. Let any $j,k\in [n]$ such that $j\ne k$ and 
	$k\ne i$. If $A=e^i(e^i)\tp$, then the inequality follows by using  inequality \eqref{eq:foc} 
	with $u=e^k$, $v=e^j$. If, $A=-e^i(e^i)\tp$, then the inequality follows by using 
	inequality \eqref{eq:foc} with $u=e^j$, $v=e^k$.  
	\medskip
		
	Conversely, suppose that $b_j\le -2\delta_{ij}$, for all $j\in [n]$.  
	Let any $u\in\Sp$ and any $v\in\Sp\cap K$ with $\lng u,v\rng=0$. If $A=e^i(e^i)^\top$, then we 
	have
	\[\f12\lng b,v\rng\le\f12b_iv_i\le -v_i\le -v_i^2\le u_i^2-v_i^2=\lng Au,u\rng-\lng Av,v\rng,\] 		 which implies inequality
	\eqref{eq:foc}. Therefore, $f$ is spherically convex. If $A=-e^i(e^i)^\top$, then inequality 
	\eqref{eq:foc} trivially holds whenever $u_i^2-v_i^2\ge 0$. Therefore, to show that 
	\eqref{eq:foc} holds, we can assume without loss of generality that $v_i^2-u_i^2>0$. Hence, 
	we have \[\f12\lng b,v\rng\le\f12b_iv_i\le -v_i\le v_i^2-u_i^2=\lng Au,u\rng-\lng Av,v\rng,\] 		 which implies inequality
	\eqref{eq:foc}. Thus, $f$ is spherically convex.
\end{proof}

Although only item (ii) of the next lemma is used in the proof of Theorem \ref{th:cd}, we have
included item (i) as well to make the result more complete.

\begin{lemma}\label{lem:ndp}
	Let $K=\R^n_+$, $i,j\in [n]$, $i\ne j$ and $A=\pm e^i(e^j)\tp\pm e^j(e^i)\tp$. 
	Then, the following statements hold: 
	\begin{enumerate}[(i)]
		\item If $f$ is spherically convex, then $b_k\le 2(\delta_{ik}+\delta_{jk}-1)$ for all 
			$k\in [n]$.
		\item If $b_k\le -3\mp 1$, for all $k\in [n]$, then $f$ is 
			spherically convex.
	\end{enumerate}
\end{lemma}

\begin{proof}
	\begin{enumerate}[(i)]
		\item  Let $f$ be spherically convex. If $k\notin\{i,j\}$, then the statement follows 
			by using inequality 
			\eqref{eq:foc} with $u=(1/\sqrt2)e^i\mp (1/\sqrt2)e^j$ and $v=e^k$. If 
			$k\in\{i,j\}$, then the inequality follows by using inequality \eqref{eq:foc} 
			with $u=e^\ell$ and $v=e^k$, where $\ell\in\{i,j\}\setminus\{k\}$.	
		\item Let $b_k\le -3\mp 1$, for all $k\in [n]$. First, suppose that 
			$A=e^i(e^j)\tp+e^j(e^i)\tp$. Let any $u\in\Sp$ and any 
			$v\in\Sp\cap K$ with $\lng u,v\rng=0$. We have
			\begin{eqnarray*}
				\lng Au,u\rng-\lng Av,v\rng=2u_iu_j-2v_iv_j
				\ge-(u_i^2+u_j^2)-(v_i^2+v_j^2)\ge
				-\sum_{k=1}^n u_k^2-\sum_{k=1}^n v_k^2\\
				=-2\sum_{k=1}^n v_k^2\ge -2\sum_{k=1}^n v_k=
				\f12\lf(-4\sum_{k=1}^n v_k\rg)\ge\f12\lf(\sum_{k=1}^n b_kv_k\rg)
				=\f12\lng b,v\rng,
			\end{eqnarray*}
			which implies inequality \eqref{eq:foc}. Therefore, $f$ is spherically 
			convex. Next, suppose that $A=-e^i(e^j)\tp+e^j(e^i)\tp$. Let any $u\in\Sp$ 
			and any $v\in\Sp\cap K$ with $\lng u,v\rng=0$. We have 
			\begin{eqnarray*}
				\lng Au,u\rng-\lng Av,v\rng=2v_iv_j-2u_iu_j
				\ge-2u_iu_j\ge-(u_i^2+u_j^2)\ge
				-\sum_{k=1}^n u_k^2
				\\=-\sum_{k=1}^n v_k^2\ge -\sum_{k=1}^n v_k=
				\f12\lf(-2\sum_{k=1}^n v_k\rg)
				\ge\f12\lf(\sum_{k=1}^n b_kv_k\rg)=\f12\lng b,v\rng,
			\end{eqnarray*}
			which implies inequality \eqref{eq:foc}. Therefore, $f$ is spherically convex.
	\end{enumerate}
\end{proof}

The following theorem and corollary provide additional large classes of spherically convex 
functions.

\begin{theorem}\label{th:cd}
	Let $K=\R^n_+$, \[b=\sum_{i=1}^na_{ii}^+b^{i}_++\sum_{i=1}^na_{ii}^-b^{i}_-
	+\sum_{\substack{i,j=1 \\ i \neq j}}^n a_{ij}^+b_+^{ij}
	+\sum_{\substack{i,j=1 \\ i \neq j}}^n a_{ij}^-b_-^{ij},\] where 
	$b^i_+,b^i_-,b^{ij}_+,b^{ij}_-\in\R^n$ with 
	${b^i_+}_k,{b^i_-}_k\le-2\delta_{ik},$ ${b^{ij}_-}_k\le -2$, and ${b^{ij}_+}_k\le -4$, for all 
	$i,j,k\in [n]$. Then, $f$ is spherically convex. 
\end{theorem}

\begin{proof} 
	Since \[A=\sum_{i=1}^n a_{ii}^+e_ie_i\tp+\sum_{i=1}^n a_{ii}^-\lf(-e_ie_i\tp\rg)+
		\sum_{\substack{i,j=1 \\ i \neq j}}^{n} a_{ij}^+
	\lf(e_ie_j\tp+e_je_i\tp\rg)+\sum_{\substack{i,j=1 \\ i \neq j}}^{n} a_{ij}^-
\lf(-e_ie_j\tp-e_je_i\tp\rg)\] and $f=f_{A,b,c}$, we have
	\[f=\sum_{i=1}^{n} a_{ii}^+f_{e_ie_i\tp
	,b^i_+,c^i_+}+\sum_{i=1}^{n} a_{ii}^-f_{-e_ie_i\tp,b^i_-,c^i_-}
	+\sum_{\substack{i,j=1 \\ i \neq j}}^{n} a_{ij}^+f_{e_ie_j\tp
	+e_je_i\tp,b^{ij}_+,c^{ij}_+}
	+\sum_{\substack{i,j=1 \\ i \neq j}}^{n} a_{ij}^-f_{-e_ie_j\tp-e_je_i\tp,b^{ij}_-,c^{ij}_-},\] 
	where $c^i$, $c^{ij}_+$ and $c^{ij}_-$ are arbitrary real constants with
	\[c=\sum_{i=1}^na_{ii}c^{i}_++\sum_{i=1}^na_{ii}c^{i}_-
	+\sum_{\substack{i,j=1 \\ i \neq j}}^n a_{ij}^+c_+^{ij}
	+\sum_{\substack{i,j=1 \\ i \neq j}}^n a_{ij}^-c_-^{ij}.\]
	Hence, items (ii) of Lemmas \ref{lem:d}-\ref{lem:ndp} 
	imply that $f$ is a linear combination of spherically convex functions 
	with nonnegative coefficients. Therefore, $f$ is spherically convex.
\end{proof}

\begin{corollary}
	Suppose that $K=\R^n_+$ and $A\ne 0$. Then, the following statements hold:
	\begin{enumerate}[(i)]
		\item
			Let \[b=\sum_{i=1}^n|a_{ii}|b^{i}
			+\sum_{\substack{i,j=1 \\ i \neq j}}^n |a_{ij}|b^{ij},\] where 
			$b^i,b^{ij}\in\R^n$ with $b^i_k\le -2\delta_{ik}$ and $b^{ij}_k\le -4$, for all 
			$i,j,k\in [n]$. Then, $f$ is spherically convex.
		\item\label{mc} If \[b_k\le-2|a_{kk}|
			-4\sum_{\substack{i,j=1 \\ i \neq j}}^n a_{ij}^+
			-2\sum_{\substack{i,j=1 \\ i \neq j}}^n a_{ij}^-,\]
			for all $k\in [n]$, then $f$ is sperically convex. 
		\item
			If \[b_k\le-2|a_{kk}|
			-4\sum_{\substack{i,j=1 \\ i \neq j}}^n |a_{ij}|,\]
			for all $k\in [n]$, then $f$ is sperically convex.
	\end{enumerate}

\end{corollary}

\begin{proof}
	$\,$

	\begin{enumerate}[(i)]
		\item Since $|a_{ij}|=a_{ij}^++a_{ij}^-$, for all $i,j\in [n]$, the result follows from 
			Theorem \ref{th:cd} with $b^i_+=b^i_-=b^i$ and $b^{ij}_+=b^{ij}_-=b^{ij}$.
		\item Since $|a_{ii}|=a_{ii}^++a_{ii}^-$, for all $i\in [n]$, the result follows
			from Theorem \ref{th:cd} with 
			\[b^i_{+k}=b^i_{-k}=-2\delta_{ik}\f{b}{\ds-2\sum_{i=1}^n|a_{ii}|
			-4\sum_{\substack{i,j=1 \\ i \neq j}}^n a_{ij}^+
			-2\delta{ik}\sum_{\substack{i,j=1 \\ i \neq j}}^n a_{ij}^-},\] 
			\[b^{ij}_-=-2\f{b}{\ds-2\sum_{i=1}^n|a_{ii}|
			-4\sum_{\substack{i,j=1 \\ i \neq j}}^n a_{ij}^+
			-2\sum_{\substack{i,j=1 \\ i \neq j}}^n a_{ij}^-}\] 
			and $b^{ij}_+=2b^i_+$, for all $i,j,k\in [n]$.
		\item It follows from $|a_{ij}|=a_{ij}^++a_{ij}^-$, for all $i,j\in [n]$ and item 
			\ref{mc}.
	\end{enumerate}
\end{proof}

\section{Final remarks} In this paper we studied the geodesic convexity of
non-homogeneous quadratic functions defined on a proper spherically convex set
(i.e., a geodesically convex set on the sphere), which we simply called spherical convexity.
We presented various necessary, sufficient
and equivalent conditions for the spherical convexity of such functions. Besides
being useful for positive and negative certificates regarding their spherical
convexity, these conditions also yield intriguing consequences.
\begin{enumerate}
	\item For any symmetric matrix that defines the homogeneous quadratic part of a 
		non-homogeneous quadratic function, there exist infinitely many vectors that 
		define the homogeneous linear term of the function, ensuring that the function
		is spherically convex, making the class of spherically convex
		non-homogeneous quadratic functions much larger than the class of
		(Euclidean) convex non-homogeneous quadratic functions. 
	\item The homogeneous linear term of
		the function has no infuence on the (Euclidean) convexity of the
		function, but it is crucial for its spherical convexity. 
	\item Even more interestingly, the previous consequence sharply contrasts with the 
		fact that the class of spherically convex homogeneous functions on the
		spherically convex set formed by unit vectors of positive
		coordinates is formed by the constant functions, hence it is much
		smaller than its Euclidean counterpart. 
\end{enumerate}
We found specific explicit conditions for the spherical convexity of non-homogeneous 
quadratic functions correspondig to positive matrices, Z-matrices and diagonal matrices.

A challenge for the future is to find explicit conditions for more general
non-homogeneous quadratic functions, defined on more general spherically convex
sets and/or corresponding to more general matrices. It would be particularly
interesting to study the spherical convexity of non-homogeneous quadratic
functions on the spherically convex set defined by the intersection of the
Lorentz cone with the sphere.  

\bibliographystyle{habbrv}
\bibliography{NonHomSphConv}

\end{document}